\documentclass{article}

\usepackage{typearea}
\typearea{8}
\usepackage{amsmath}

\usepackage{amssymb}
\usepackage{graphicx} 
\usepackage{wrapfig}
\usepackage{here}
\usepackage{amsthm}
\usepackage{cases}
\usepackage{url}

\newtheorem{dfn}{Definition}
\newtheorem{lem}{Lemma}
\newtheorem{thm}{Theorem}
\newtheorem{cor}{Corollary}

\newtheorem{rmk}{Remark}
\newtheorem{prop}{Proposition}

\title{Singular function emerging from one-dimensional elementary cellular automaton Rule $150$}
\author{Akane Kawaharada\footnote{E-mail: aka@kyokyo-u.ac.jp, Postal address: 1, Fujinomoricho, Fukakusa, Fushimi-ku, Kyoto-shi, Kyoto, 612-8522, Japan} \vspace{2mm}\\
Department of Mathematics, Kyoto University of Education}

\begin{document}

\maketitle

\begin{abstract}
In this paper, we give a singular function on the unit interval derived from the dynamic of the one-dimensional elementary cellular automaton Rule $150$.
We describe properties of the resulting function, that is strictly increasing, uniformly continuous, and differentiable almost everywhere, and 
we show that it is not differentiable at dyadic rational points.
We also give functional equations that the function satisfies, 
and show that the function is the only solution of the functional ones.
\end{abstract}

\hspace{2.5mm} {\it Key words} : cellular automaton, fractal, singular function

\section{Introduction}

There exist many pathological functions. 
The Weierstrass function and the Takagi function, for example,   
are real-valued functions that are continuous everywhere but nowhere differentiable  \cite{weier1872, takagi1903}. 
Generalized results of the Takagi function were given in \cite{hatayama1984}.
Okamoto's function is a one-parameter family of self-affine functions whose differentiability is determined by the parameter; it is differentiable almost everywhere, non-differentiable almost everywhere, or nowhere differentiable \cite{okamoto2005, okamoto2007, kobayashi2009}.
A singular function is defined by monotonically increasing (or decreasing), continuous everywhere,  and has zero derivative almost everywhere. 
The Cantor function is an example of a singular function \cite{cantor1884}, 
that is also referred to as the Devil's staircase, and there are infinite number of steps in $[0,1]$ while it is constant most of them.
Salem's function is a self-affine function, that is another example of a singular function \cite{salem1943, derham1957, ulam1934, yhk1997}.
There are several works discussed the relationship between the function and cellular automata.
For the one-dimensional elementary cellular automaton Rule $90$ the limit set is characterized by Salem's function \cite{kawanami2014}, and
for a two-dimensional automaton that is a mathematical model of a crystalline growth the limit set is also characterized by Salem's one (numerical result was given in \cite{kawa2014a}, proofs were given in \cite{kawanami2017, kawanami2019}).
In the case of these previous works, the number of nonzero states in a spatial or spatio-temporal pattern of a cellular automaton is represented by functional equations that are equal to functional ones of Salem's singular function.

In this paper, we provide a new singular function by the elementary cellular automaton Rule $150$.
Figure~\ref{fig:150up} shows a spatio-temporal pattern of Rule $150$ from time step $0$ to $31$, and Figure~\ref{fig:150lim} shows its limit set.
In this case, the number of nonzero states in the spatial pattern can not be represented by simple functional equations, and we can not use the same constructing method of the previous works.
For Rule $150$ the authors have calculated the number of nonzero states in the spatial and spatio-temporal pattern \cite{kawanami2020}.
Thus, by normalizing and limiting the dynamic of the numbers, we provide a size of a self-similar set, and write down the function by an infinite sum of the sizes of the self-similar sets.
We show that the resulting function is a singular function, and the function is not differentiable at dyadic rational points.
Functional equations that the function satisfies are also given.

The remainder of the paper is organized as follows. 
Section \ref{sec:pre} describes the preliminaries concerning the cellular automaton Rule $150$ and the number of nonzero states in its spatial and spatio-temporal patterns. 
In Section \ref{sec:main}, we provide a definition of the given function and write it down by using self-similarities of the spatio-temporal pattern of Rule $150$. 
We show that the resulting function is a singular function, and 
give functional equations that the function satisfies.
Lastly, Section \ref{sec:cr} discusses the findings of this paper and describes possible areas for future studies.

\section{Preliminaries}
\label{sec:pre}

In this section, we present some definitions and notations for elementary cellular automata and their limit sets.
We also provide an overview of previous results about the number of nonzero states in spatial or spatio-temporal patterns of cellular automata.

\subsection{One-dimensional elementary cellular automaton Rule 150 and its limit set}
 
Let $\{0,1\}$ be a binary state set and $\{0,1\}^{\mathbb Z}$ be the one-dimensional configuration space.
Suppose that $(\{0,1\}^{\mathbb Z}, T)$ is a discrete dynamical system consisting of a space $\{0,1\}^{\mathbb Z}$ and a transformation $T$ on $\{0,1\}^{\mathbb Z}$.
The $n$-th iteration of $T$ is denoted by $T^n$. 
Thus, $T^0$ is the identity map.

\begin{dfn}
{\bf A one-dimensional elementary cellular automaton} $(\{0, 1\}^{\mathbb Z}, T)$ is given by 
\begin{align}
(T x)_{i} = f(x_{i-1}, x_i, x_{i+1}) 
\end{align}
for $i \in {\mathbb Z}$ and $x = \{x_i\}_{i \in {\mathbb Z}} \in \{0, 1\}^{\mathbb Z}$,  
where $f : \{0,1\}^{3} \to \{0,1\}$ is a map depending on the nearest three states. 
We call $f$ a local rule of $T$. 
\end{dfn}
This is the simplest nontrivial cellular automaton. 
This class includes $256$ automata, referred to by the Wolfram code from Rule $0$ to Rule $255$.
For each state $x_i$ ($i \in {\mathbb Z}$), the next state $(T x)_i$ is determined by the nearest three states $(x_{i-1}, x_i ,x_{i+1})$. 
In the case of Table \ref{tab:1dECA}, the Wolfram code is Rule $150$, 
because $1 \cdot 2^7 + 0 \cdot 2^6 + 0 \cdot 2^5 + 1 \cdot 2^4 + 0 \cdot 2^3 + 1 \cdot 2^2 + 1 \cdot 2^1 + 0 \cdot 2^0 = 150$.
{\bf The local rule of Rule $150$} $(\{0, 1\}^{\mathbb Z}, T_{150})$ is also given by
\begin{align}
\label{eq:1d150}
(T_{150} x)_i &= x_{i-1} + x _{i} +x_{i+1} \quad \mbox{(mod $2$)},
\end{align} 
for $x \in \{0,1\}^{\mathbb Z}$.
The local rules given in Table~\ref{tab:1dECA} and Equation~(\ref{eq:1d150}) are mathematically equal.

\begin{table}[htb]
\caption{Local rule of Rule $150$}
\begin{center}
\begin{tabular}{c | c c c c c c c c}
\hline 
$x_{i-1} x_i x_{i+1}$ & $111$ & $110$ & $101$ & $100$ & $011$ & $010$ & $001$ & $000$ \\ 
\hline  
$(T_{150} x)_i$ & $1$ & $0$ & $0$ & $1$ & $0$ & $1$ & $1$ & $0$ \\
\hline
\end{tabular}
\label{tab:1dECA}
\end{center}
\end{table}

The configuration $x_o \in \{0,1\}^{\mathbb Z}$ is called {\bf the single site seed}, wherein  
\begin{align}
({x_o})_i = \left\{
\begin{array}{ll}
1 & \mbox{if $i = 0$},\\
0 & \mbox{if $i \in {\mathbb Z} \backslash \{ 0 \}$}.
\end{array}
\right. 
\label{eq:sss}
\end{align} 
Figure \ref{fig:150up} shows the orbit of Rule $150$ from the single site seed $x_o$ as an initial configuration until time step $2^5-1$.

\begin{figure}[htbp]
\begin{minipage}{0.5\hsize}
\begin{center}
\includegraphics[width=.9\linewidth]{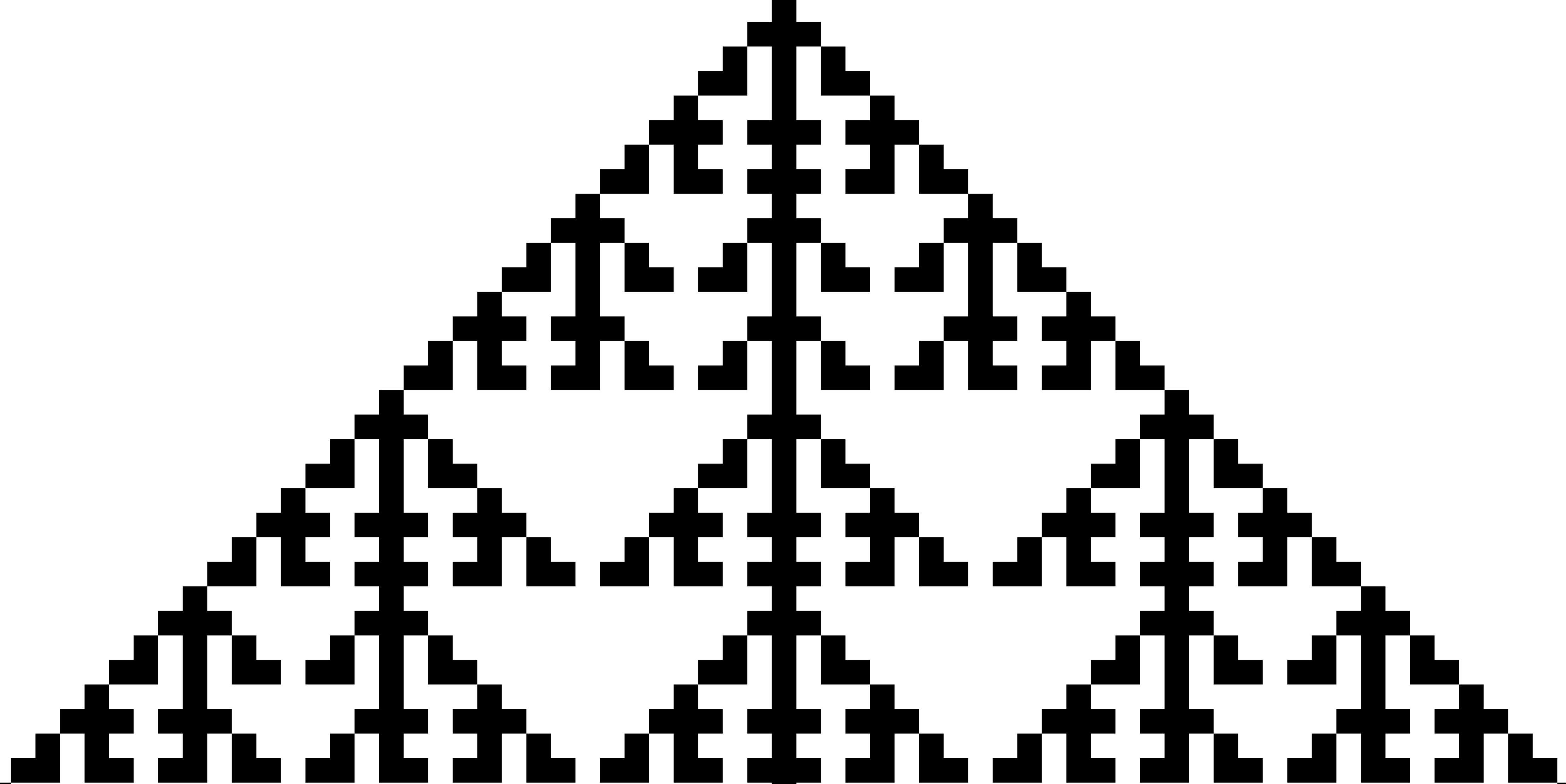}
\end{center}
\caption{Spatio-temporal pattern of Rule $150$, $\{T_{150} x_o\}_{n=0}^{31}$}
\label{fig:150up}
\end{minipage}
\begin{minipage}{0.5\hsize}
\begin{center}
\includegraphics[width=.9\linewidth]{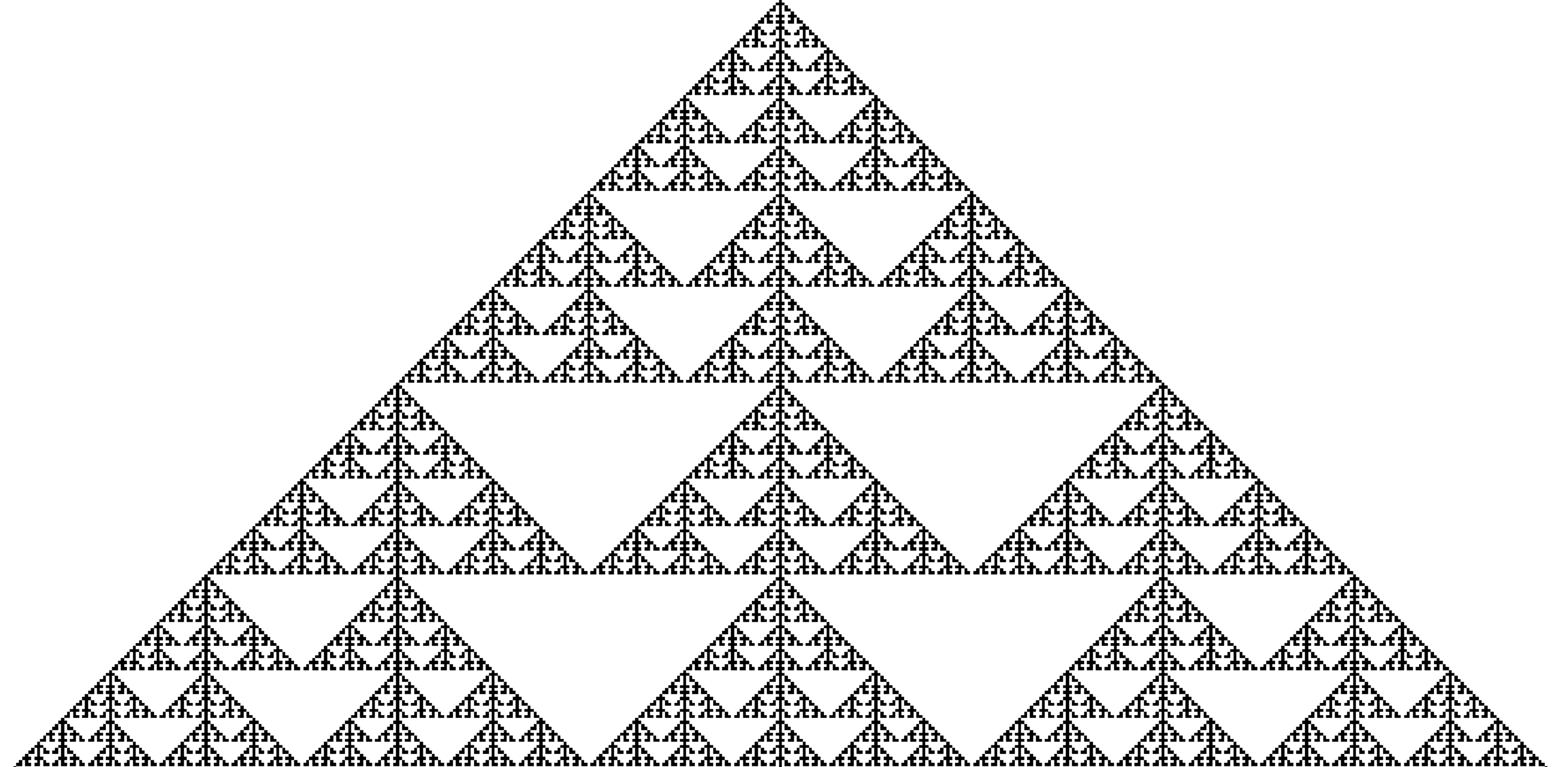}
\end{center}
\caption{Limit set of Rule $150$ from the single site seed $x_o$}
\label{fig:150lim}
\end{minipage}
\end{figure}

Suppose that $\{ T^n x_o \}$ is a dynamic of a cellular automaton from the single site seed, and a subset of a two-dimensional Euclidean space $S(n)$ is given by 
\begin{align}
S(n) = \{ (i, t) \in {\mathbb Z}^2 \mid (T^t x_o)_i > 0, 0 \leq t \leq n \},
\end{align}
that consists of nonzero states from time step $0$ until $n$.
{\bf A limit set} of a cellular automaton is defined by 
$\lim_{n \to \infty} (S(n)/n)$,
if it exists, where $S(n)/n$ is a contracted set of $S(n)$ with a contraction rate of $1/n$.
Before evaluating the limit, $S(n)/n$ for finite $n$ is called {\bf a prefractal set} if the limit set exists. 
For limit sets of linear cellular automata the following two theorems have been shown.

\begin{thm}[\cite{takahashi1992}]
\label{thm:tkhs01}
Consider a $p^m$-state linear cellular automaton ($p$ is a prime number, $m \in {\mathbb Z}_{>0}$). 
If $p^{m-1}$ divides time step $n$, then $(T^{pn} x_o)_{pi} = (T^n x_o)_i$. 
If $p^m$ divides $n$ and at least one of the elements of $i$ is indivisible by $p$, 
then $(T^n x_o)_i$ equals $0$.
\end{thm}

\begin{thm}[\cite{takahashi1992}]
\label{thm:tkhs02}
For a $p^m$-state linear cellular automaton ($p$ is a prime number, $m \in {\mathbb Z}_{>0}$) 
its limit set $\lim_{k \to \infty} (S(p^k-1)/p^k)$ exists.
\end{thm}

Based on Theorems~\ref{thm:tkhs01} and \ref{thm:tkhs02}, we obtain the following corollary, because Rule $150$ is a two-state linear cellular automaton.

\begin{cor}
\label{cor:limset150}
A subset of a two-dimensional Euclidean space $S_{150}(2^k-1)$ is given by 
\begin{align}
S_{150}(2^k-1) = \{ (i, t) \in {\mathbb Z}^2 \mid (T_{150}^t x_o)_i > 0, 0 \leq t \leq 2^k-1 \}.
\end{align}
The limit set of the orbit of Rule $150$ from the single site seed $x_o$, 
$\lim_{k \to \infty} (S_{150}(2^k-1)/2^k)$, is a fractal 
whose Hausdorff dimension is $\log (1+\sqrt{5}) / \log 2$.
\end{cor}
Figure \ref{fig:150lim} shows the limit set of Rule $150$ 
with time steps $n=2^k-1$ as $k$ tends to infinity.

\subsection{Numbers of nonzero states for Rule $150$}

Let $num_T(n)$ be the number of nonzero states in a spatial pattern $T^n x_o$ for time step $n$, and $cum_T(n)$ be the cumulative sum of the number of nonzero states in a spatial pattern $T^m x_o$ from time step $m=0$ to $n$ for a cellular automaton. Thus, 
\begin{align}
num_T(n) = \sum_{i \in {\mathbb Z}} (T^n x_o)_i, \
cum_T(n) = \sum_{m = 0}^n \sum_{i \in {\mathbb Z}} (T^m x_o)_i.
\end{align}
In the case of Rule $150$, the numbers are denoted by $num_{150}(n)$ and $cum_{150}(n)$, respectively.
Figure~\ref{fig:cum150} shows the dynamic of $cum_{150}(n)$ for $0 \leq n < 256$.
We introduce the previous results about the number of nonzero states in the spatio-temporal and spatial patterns according to self-similar structures.

\begin{figure}[H]
\begin{center}
\includegraphics[width=.47\linewidth]{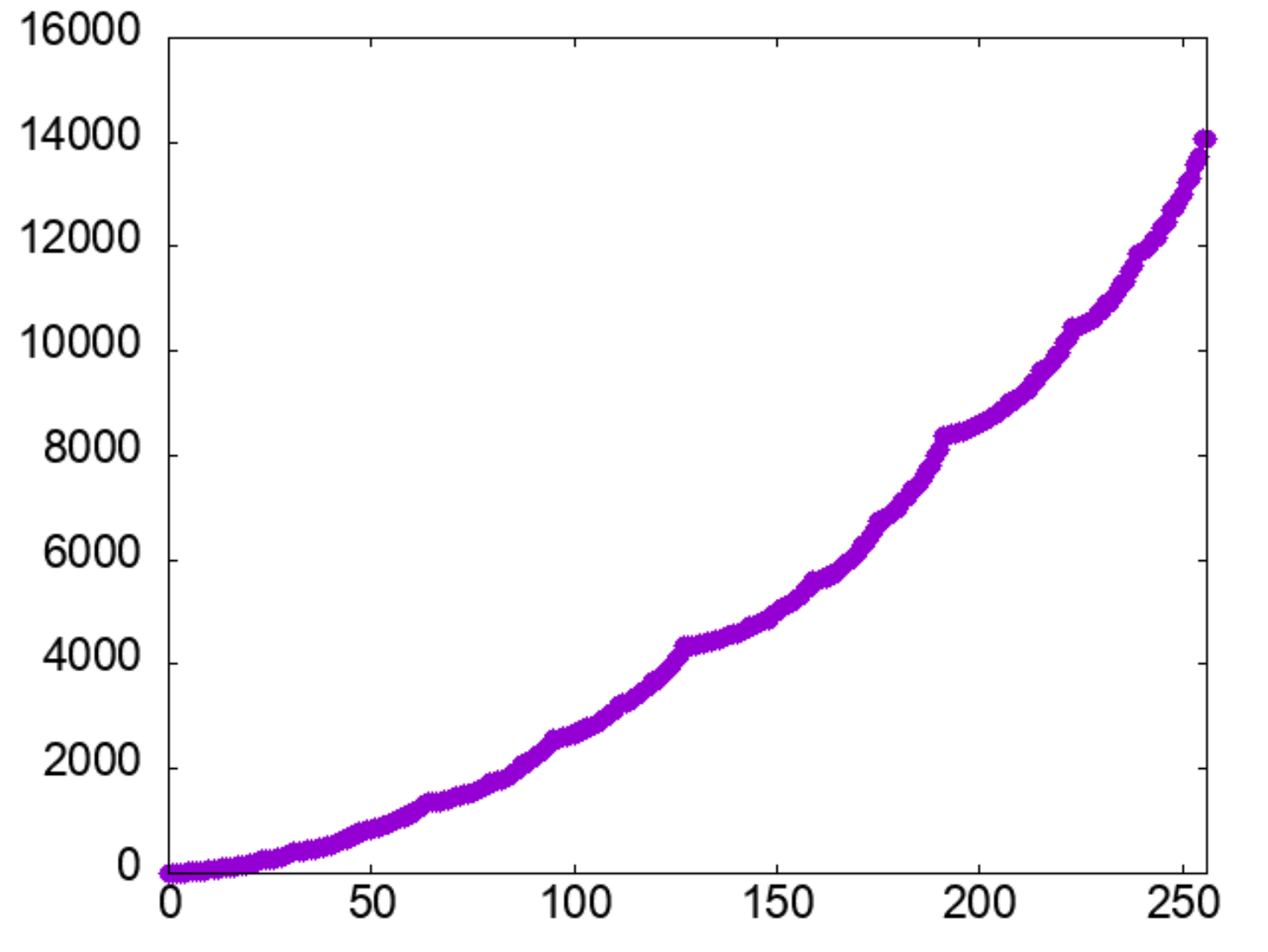}
\end{center}
\caption{$\{cum_{150}(n)\}$ of Rule $150$ for $0 \leq n < 256$}
\label{fig:cum150}
\end{figure}

\begin{prop}[\cite{kawanami2020, takahashi1992, claussen2008}]
\label{prop:cum150}
We introduce the transition matrix $M$ and the vector $v_0$ as the initial values;
\begin{align}
M = 
\left(
\begin{array}{cc}
2 & 4 \\
1 & 0 
\end{array} 
\right), \
v_0 = 
\left(
\begin{array}{c}
4 \\
1  
\end{array}
\right).
\end{align}
The cumulative sum of the number of nonzero states from time step $0$ to $2^k-1$ for Rule $150$, $cum_{150}(2^k-1)$, is given by $a M^{k-1} v_0$ for a vector $a = (1 \ 0)$ and $k \geq 1$. 
Hence, 
\begin{align}
\label{eq:cum150}
cum_{150}(2^k -1) 
&= \frac{\sqrt{5}}{20}(1+\sqrt{5})^{k+2}-\frac{\sqrt{5}}{20}(1-\sqrt{5})^{k+2}.
\end{align}
\end{prop}

\begin{prop}[\cite{kawanami2020}]
\label{prop:num}
We introduce the transition matrices and the vector
\begin{align}
M_0 =
\left( 
\begin{array}{cc}
1 & 0 \\
1 & 0 
\end{array}
\right) , \
M_1 = 
\left(
\begin{array}{cc}
1 & 2 \\
1 & 0
\end{array}
\right), \
u_0 =
\left(
\begin{array}{c}
1 \\
1  
\end{array}
\right).
\end{align}
Assuming that the binary expansion of $n$ is $n_{l-1} n_{l-2} \cdots n_1 n_0$ $(n_i \in \{0, 1\}, i=0,1, \ldots, l-1)$,
let $p_r$ be the number of clusters consisting of continuous $r$ $1$s in the binary number. Thus,
\begin{align}
\label{eq:num150}
num_{150}(n) &=  a M_{n_{l-1}} M_{n_{l-2}} \cdots M_{n_0} u_0
= \prod_{r=0}^l (a M_1^r u_0)^{p_r}  \nonumber\\
&= \prod_{r=0}^l \left( \frac{2^{r+2} + (-1)^{r+1}}{3} \right)^{p_r}.
\end{align}
\end{prop}

\begin{rmk}
In this paper we set $num_{150}(-1) = cum_{150}(-1) = 0$ for $n = -1$, due to a technical reason.
\end{rmk}

\section{Main results}
\label{sec:main}

In this section, we give a function on the unit interval by the cumulative sum of the number of nonzero states of Rule $150$.

\begin{dfn}
\label{dfn:fx0}
For $x = \sum_{i=1}^{\infty} (x_i/2^i) \in [0,1]$ and $k \in {\mathbb Z}_{>0}$ a function $F_k(x)$ is given by $cum_{150} ((\sum_{i=1}^k x_i 2^{k-i})-1 ) / cum_{150}(2^k-1)$, that is a normalized sum of the number by $cum_{150}(2^k-1)$. 
Considering its limit, we name it $F$. 
Thus, 
\begin{align}
\label{eq:org}
F (x) &:= \lim_{k \to \infty} F_k (x) 
= \lim_{k \to \infty} \frac{cum_{150}\left((\sum_{i=1}^k x_i 2^{k-i})-1 \right)}{cum_{150}(2^k-1)}.
\end{align}
\end{dfn}

Figure~\ref{fig:plots} shows the graph of $F(x)$ for $x \in [0,1]$ and the limit set of Rule $150$.

\begin{figure}[H]
\begin{center}
\includegraphics[width=.55\linewidth]{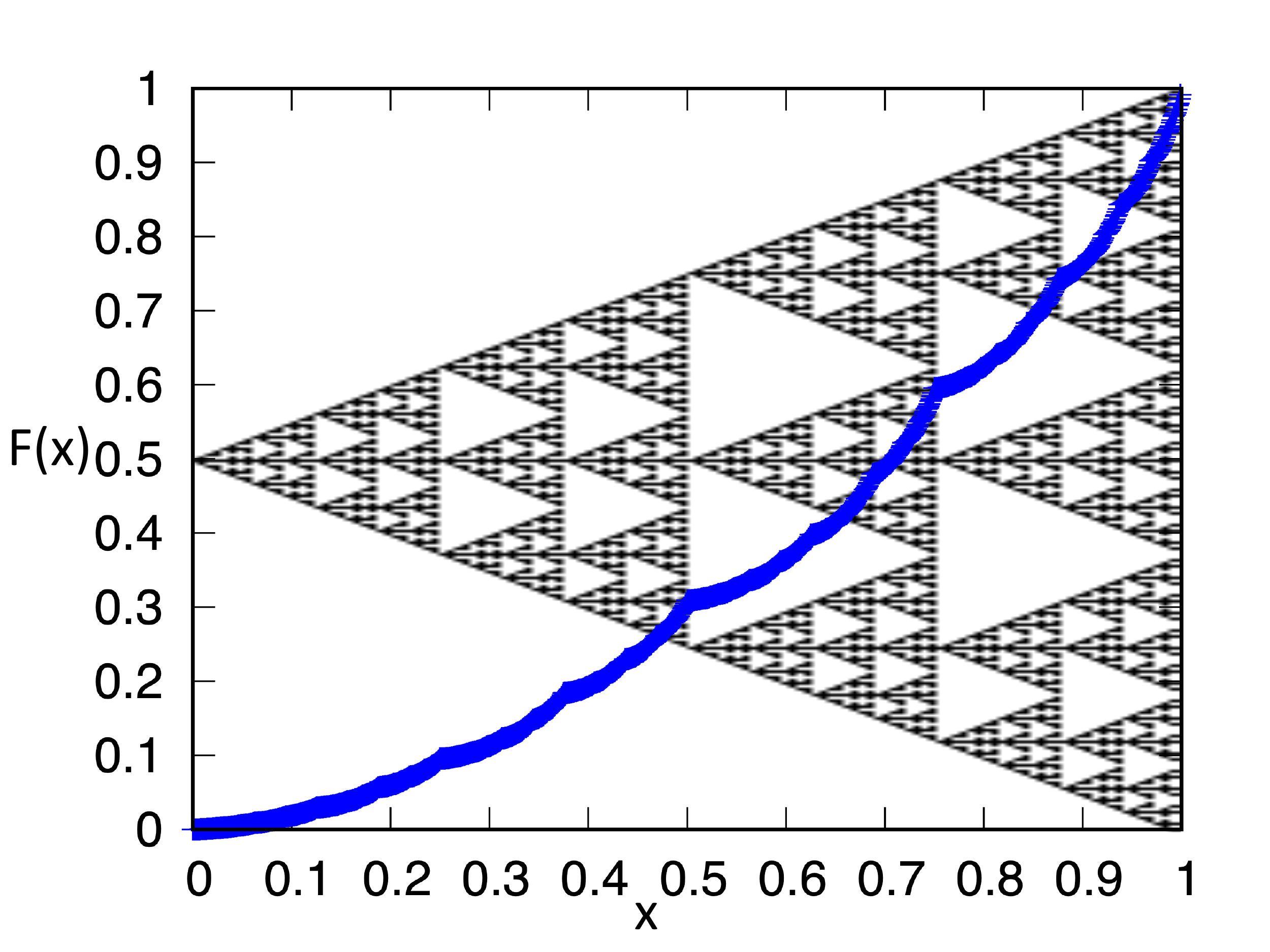}
\end{center}
\caption{$F(x)$ and the limit set of Rule $150$}
\label{fig:plots}
\end{figure}

In Section~\ref{subsec:def_sf}, we construct the function $F$ focusing on self-similar structures of the spatio-temporal pattern $\{ T_{150} x_o \}$, 
and Section~\ref{subsec:prp_sf} describes properties of $F$.
In Section~\ref{subsec:sfunc_sf} we give functional equations that $F$ satisfies, and we prove that $F$ is the only solution of them.

\subsection{Constructing the function $F$}
\label{subsec:def_sf}

First, we show that $cum_{150}(m-1)$ for $m>0$ is represented by a sum of $cum_{150}(2^i-1)$s. 
By Theorems~\ref{thm:tkhs01} and \ref{thm:tkhs02} we obtain the following result.

\begin{lem}
\label{lem:finite}
Let $x = \sum_{i=1}^{\infty} x_i/2^i\in [0,1]$, and $m = \sum_{i=0}^{k-1} x_{k-i} \, 2^i = \sum_{i=1}^k x_i \, 2^{k-i} \geq 0$.
\begin{align}
cum_{150}(m-1) 
&= \sum_{i=1}^k x_i \ num_{150}\left( \sum_{j=1}^{k-i-1} x_j 2^{k-i-1-j} \right) cum_{150}(2^{k-i}-1).
\end{align}
\end{lem}

Lemma~\ref{lem:finite} gives the accurate number of nonzero states for any time step $m > 0$.
Thus, the function $F$ is given by the following equation.

\begin{thm}
\label{thm:fx}
Let 
\begin{align}
a = 
\left(
\begin{array}{cc}
1 & 0 
\end{array}
\right), \
M_0 = 
\left(
\begin{array}{cc}
1 & 0 \\
1 & 0 
\end{array}
\right), \
M_1 =
\left(
\begin{array}{cc}
1 & 2 \\
1 & 0 
\end{array}
\right), \
u_0 =
\left(
\begin{array}{c}
1 \\
1 
\end{array}
\right).
\end{align}
For $x = \sum_{i=1}^{\infty} ( x_i / 2^i ) \in [0,1]$ the function $F : [0,1] \to [0,1]$ is given by
\begin{align}
\label{eq:def}
F(x) = \sum_{i=1}^{\infty} x_i r(x)_i \alpha^i, 
\end{align}
where $r(x)_i = a M_{x_{i-1}} M_{x_{i-2}} \cdots M_{x_2} M_{x_1} M_{x_0} u_0$ and $\alpha = (\sqrt{5}-1)/4$.
\end{thm} 
 
\begin{proof}
By Propositions~\ref{prop:cum150}, \ref{prop:num}, and Lemma~\ref{lem:finite} we have
\begin{align}
F_k (x) 
&= \frac{cum_{150}\left((\sum_{i=1}^k x_i 2^{k-i})-1 \right)}{cum_{150}(2^k-1)}\\
&= \frac{\sum_{i=1}^k x_i \ num\left( \sum_{j=1}^{k-i-1} x_j 2^{k-i-1-j} \right) cum(2^{k-i}-1)}{cum_{150}(2^k-1)}\\
&= \frac{\sum_{i=1}^k x_i \ ( a M_{x_0} M_{x_1} \cdots M_{x_{i-1}} u_0 ) \left( \sqrt{5} \left((1+\sqrt{5})^{k+2-i}-(1-\sqrt{5})^{k+2-i}\right)/20 \right)}{\sqrt{5} \left((1+\sqrt{5})^{k+2}- (1-\sqrt{5})^{k+2}\right)/20}\\
&= \sum_{i=1}^k x_i \, r(x)_i \left( \frac{\alpha^i}{1-((\sqrt{5}-3)/2)^{k+2}} - \frac{(1-\sqrt{5})^{-i}}{(-(\sqrt{5}+3)/2)^{k+2}-1} \right)\\
&= \left( \frac{1}{1-((\sqrt{5}-3)/2)^{k+2}} \sum_{i=1}^k x_i \, r(x)_i \alpha^i \right) \nonumber \\
& \quad \qquad - \left( \frac{(5/2)^k}{(-(\sqrt{5}+3)/2)^{k+2}-1} \sum_{i=1}^k x_i \, r(x)_i \left(-\frac{\sqrt{5}+1}{10} \right)^i \left( \frac{2}{5} \right)^{k-i} \right). \label{eq:conv}
\end{align}

For the first term of Equation~(\ref{eq:conv}) we have $x_i \, r(x)_i \alpha^i \leq 3^{i-1} \alpha^i$ for any $i > 0$.
As $\lim_{i \to \infty} \left| (3^i \alpha^{i+1})/(3^{i-1}\alpha^i)\right| = 3 \alpha < 1$, the infinite series $\sum_{i=1}^{\infty} 3^{i-1} \alpha^i$ absolutely converges.
Thus, $\sum_{i=1}^{\infty} x_i r(x)_i \alpha^i$ also absolutely converges.
For the second terms of Equation~(\ref{eq:conv})
for any $i>0$ we have $x_i \, r(x)_i \left(-(\sqrt{5}+1)/10 \right)^i \left( 2/5 \right)^{k-i} \leq 3^{i-1} \left(-(\sqrt{5}+1)/10 \right)^i$.
Since $\lim_{i \to \infty} \left| (3^i (-(\sqrt{5}+1)/10)^{i+1})/(3^{i-1} (-(\sqrt{5}+1)/10)^i) \right| = \left| 3 (-(\sqrt{5}+1)/10) \right| < 1$, 
$\sum_{i=1}^{\infty} 3^{i-1} (-(\sqrt{5}+1)/10)^i)^i$ absolutely converges, and also $\sum_{i=1}^{\infty} x_i \, r(x)_i \left(-(\sqrt{5}+1)/10 \right)^i \left( 2/5 \right)^{k-i}$ absolutely converges.
We note that $\lim_{k \to \infty} 1/(1-((\sqrt{5}-3)/2)^{k+2}) = 1$, and $\lim_{k \to \infty} (5/2)^k / ((-(\sqrt{5}+3)/2)^{k+2}-1)=0$.
Therefore, we obtain the following result.

\begin{align}
\lim_{k \to \infty} F_k (x) 
&= \left( 1 \cdot \sum_{i=1}^{\infty} x_i \, r(x)_i \alpha^i \right) - \left( 0 \cdot  \sum_{i=1}^k x_i \, r(x)_i \left(-\frac{\sqrt{5}+1}{10} \right)^i \left( \frac{2}{5} \right)^{k-i} \right)\\
&=\sum_{i=1}^{\infty} x_i \, r(x)_i \alpha^i.
\end{align}

\end{proof}

\begin{cor}
\label{cor:fx}
By Proposition~\ref{prop:num} for $x = \sum_{i=1}^{\infty} (x_i/2^i) \in [0,1]$
we have $N_i$ and $k_j$ $(j=1, \ldots, N_i)$ such that 
$a M_{x_{i-1}} M_{x_{i-2}} \cdots M_{x_2} M_{x_1} M_{x_0} u_0 = \prod_{j=1}^{N_i} (a M_1^{k_j} u_0)$. 
Because $r(x)_i$ is represented by $\prod_{j=1}^{N_i} (a M_1^{k_j} u_0)$,
the function $F$ is also given by
\begin{align}
F(x) = \sum_{i=1}^{\infty} \left( x_i \alpha^i \prod_{j=1}^{N_i} \frac{2^{k_j+2} + (-1)^{k_j+1}}{3} \right).
\end{align}
\end{cor}

\begin{rmk}
\label{rmk:01}
We verify that $F(0)=0$ and $F(1)=1$.
By the definition of $F$ 
\begin{align}
F(0) &= F \left( \sum_{i=1}^{\infty} \frac{0}{2^i} \right) = 0,\\
F(1) &= F \left( \sum_{i=1}^{\infty} \frac{1}{2^i} \right) 
= \sum_{i=1}^{\infty} r(x)_i \alpha^i = \sum_{i=1}^{\infty} b_i \alpha^i =1,
\end{align}
where $b_i := a M_1^{i-1} u_0 = (2^{i+1}+(-1)^i)/3$.
\end{rmk}

\begin{rmk}
The binary expansion of $x$ is unique except for dyadic rationals 
$x= m/2^i$, which have two possible expansions. 
We check that the definition of $F$ is consistent for the numbers 
having two binary expansions.
Let $x = \sum_{i=1}^k (x_i / 2^i) + 1 / 2^{k+1}$ and 
$y = \sum_{i=1}^k ( x_i / 2^i) + \sum_{i=k+2}^{\infty} (1/2^i)$ for $x_i \in \{0,1\}$ and $k \in {\mathbb Z}_{>0}$. Thus, $x=y$.
Here we confirm that $F(x)=F(y)$.
By the definition of $F$
\begin{align}
F(x) &= F \left( \sum_{i=1}^k \frac{x_i}{2^i} + \frac{1}{2^{k+1}} \right) 
= \left(\sum_{i=1}^k x_i r(x)_i \alpha^i \right) + r(x)_{k+1} \alpha^{k+1},\\
F(y) &= F \left( \sum_{i=1}^k \frac{x_i}{2^i} + \sum_{i=k+2}^{\infty} \frac{1}{2^i}\right) 
= \left( \sum_{i=1}^{k} x_i r(x)_i \alpha^i \right) + \sum_{i=k+2}^{\infty} r(y)_i \alpha^i.
\end{align}
Set $\tilde{M_{x_k}} = M_{x_k} M_{x_{k-1}} \cdots M_{x_1} M_{x_0}$. We have
\begin{align}
F(x) - F(y) 
&= r(x)_{k+1} \alpha^{k+1} - \sum_{i=k+2}^{\infty} r(y)_i \alpha^i\\
&= ( a \tilde{M_{x_k}} u_0 ) \alpha^{k+1} - \sum_{i=k+2}^{\infty} ( a M_1^{i-k-2} M_0 \tilde{M_{x_k}} u_0 )  \alpha^i\\
&= ( a \tilde{M_{x_k}} u_0 ) \alpha^{k+1} - \sum_{i=k+2}^{\infty} ( a M_1^{i-k-2} u_0)(a \tilde{M_{x_k}} u_0 )  \alpha^i\\
&= ( a \tilde{M_{x_k}} u_0 ) \alpha^{k+1} \left( 1 - \sum_{i=1}^{\infty} ( a M_1^{i-1} u_0) \alpha^i \right)\\
&= ( a \tilde{M_{x_k}} u_0 ) \alpha^{k+1} \left( 1 - \sum_{i=1}^{\infty} b_i \alpha^i \right) = 0.
\end{align}
\end{rmk}

\subsection{Properties of the function $F$}
\label{subsec:prp_sf}

In this section we describe properties of the function $F$ given in Theorem~\ref{thm:fx}.

\begin{thm}
\label{thm:prop}
The function $F$ on $[0, 1]$ holds the following properties.
\begin{enumerate}
\item $F$ is strictly increasing,
\item $F$ is uniformly continuous,
\item $F$ is differentiable with derivative zero almost everywhere, and
\item $F$ is a singular function. 
\end{enumerate}
\end{thm}

\begin{proof}[Proof of Theorem~\ref{thm:prop} (i)]
Assuming that $y > x$, we have $k \in {\mathbb Z}_{\geq 0}$ such that $y_i=x_i$ for $\forall i \leq k$, $y_{k+1}=1$, and $x_{k+1}=0$.

We consider the following three cases 
(excepting the case of $x=\sum_{i=1}^k (x_i/2^i) + \sum_{i=k+2}^{\infty} (1/2^i)$ and
$y=\sum_{i=1}^k (x_i/2^i) + 1/2^{k+1}$, because it means $x=y$).

\begin{enumerate}
\item[(a)] Suppose that 
$x=\sum_{i=1}^k (x_i/2^i) + \sum_{i=k+2}^{\infty} (1/2^i)$ and 
$y= (\sum_{i=1}^k (x_i/2^i)) + 1/2^{k+1} + (\sum_{i=k+2}^{\infty} (y_i/2^i))$, 
where $\sum_{i=k+2}^{\infty} y_i > 0$.

By the definition $x$ is represented by $\sum_{i=1}^k (x_i/2^i) + 1/2^{k+1}$.
Let $l = \max \{ i \mid x_j = y_j \mbox{ for $\forall j \leq i$}\}$.
The number $l$ always exists, because $y \neq x$, and 
the following inequality is obtained.
\begin{align}
F(y) - F(x) &= \sum_{i=l+1}^{\infty} y_i r(y)_i \alpha^i \geq r(y)_{l+1} \alpha^{l+1} > 0,
\end{align}
since $\sum_{i=k+2}^{\infty} y_i > 0$.

\item[(b)] Suppose that 
$x=\sum_{i=1}^k (x_i/2^i) + \sum_{i=k+2}^{\infty} (x_i/2^i)$ and 
$y=\sum_{i=1}^k (x_i/2^i) + 1/2^{k+1}$, 
where $\prod_{i=k+2}^{\infty} x_i = 0$.

By the definition $y$ is represented by $\sum_{i=1}^k (x_i/2^i) + \sum_{i=k+2}^{\infty} (1/2^i)$.
Since $\prod_{i=k+2}^{\infty} x_i = 0$, we have the following inequality.
\begin{align}
F(y) - F(x) &= \left( \sum_{i=k+2}^{\infty} r(y)_i \alpha^i \right) - 
\left( \sum_{i=k+2}^{\infty} x_i r(x)_i \alpha^i \right)\\
&= \left( \sum_{i=1}^{\infty} r(y)_{i+k+1} \alpha^{i+k+1} \right) - 
\left( \sum_{i=1}^{\infty} x_{i+k+1} r(x)_{i+k+1} \alpha^{i+k+1} \right)\\
&= \left( r(x)_{k+1} \alpha^{k+1} \sum_{i=1}^{\infty} b_i \alpha^i \right) - 
\left( r(x)_{k+1} \alpha^{k+1} \sum_{i=1}^{\infty} x_{i+k+1} r(x)_i \alpha^i \right)\\
&= r(x)_{k+1} \alpha^{k+1} \left( 1 - \sum_{i=1}^{\infty} x_{i+k+1} r(x)_i \alpha^i \right) > 0.
\end{align}

\item[(c)] Suppose that 
$x=\sum_{i=1}^k (x_i/2^i) + \sum_{i=k+2}^{\infty} (x_i/2^i)$ and 
$y=(\sum_{i=1}^k (x_i/2^i)) + 1/2^{k+1} + (\sum_{i=k+2}^{\infty} (y_i/2^i))$, 
where $\prod_{i=k+2}^{\infty} x_i = 0$ and $\sum_{i=k+2}^{\infty} y_i > 0$.

Let $m = \min \{ i \mid y_i = 1, i > k+1 \}$. 
Thus, we have
\begin{align}
F(y) - F(x) &= \left( \sum_{i=k+1}^{\infty} y_i r(y)_i \alpha^i \right) - 
\left( \sum_{i=k+2}^{\infty} x_i r(x)_i \alpha^i \right)\\
&= r(x)_{k+1} \alpha^{k+1} + \sum_{i=k+2}^{\infty} ( y_i r(y)_i - x_i r(x)_i ) \alpha^i\\
&> r(x)_{k+1} \alpha^{k+1} + r(y)_m \alpha^m - \sum_{i=k+2}^{\infty} r(x)_i \alpha^i \label{eq:ceq}\\
&= r(x)_{k+1} \alpha^{k+1} + r(y)_m \alpha^m - r(x)_{k+1} \alpha^{k+1} \sum_{i=1}^{\infty} b_i \alpha^i\\
&= r(y)_m \alpha^m > 0.
\end{align}
The inequality~(\ref{eq:ceq}) is satisfied, because $\prod_{i=k+2}^{\infty} x_i = 0$.
\end{enumerate}
Therefore, if $y>x$, then $F(y) > F(x)$.
\end{proof}

\begin{proof}[Proof of Theorem~\ref{thm:prop} (ii)]
Let 
$x = \sum_{i=1}^{\infty} (x_i/2^i) \in [0,1]$,
$y = \sum_{i=1}^{\infty} (y_i/2^i) \in [0,1]$ ($x_i, y_i \in \{0,1\}$), 
and $x \neq y$.
We have $k \in {\mathbb Z}_{\geq 0}$ such that $x_i = y_i$ for $\forall i \leq k$.
Hence, 
\begin{align}
| y - x | \leq \frac{1}{2^k} < (3 \alpha)^k =: \frac{1-3\alpha}{\alpha} \epsilon. 
\end{align}
Without loss of generality, we assume that $y > x$. 
\begin{align}
| F(y) - F(x) | &= \left| \sum_{n=k+1}^{\infty} ( y_n r(y)_n - x_n r(x)_n ) \alpha^n \right|\\
&\leq \sum_{n=k+1}^{\infty} r(y)_n \alpha^n\\
&< \sum_{n=k+1}^{\infty} 3^{n-1} \alpha^n\\
&= \frac{\alpha}{1-3 \alpha} (3 \alpha)^k = \epsilon.
\end{align}
Since $F$ is a function on a finite bounded section $[0,1]$, $F$ is uniformly continuous.
\end{proof}

\begin{proof}[Proof of Theorem~\ref{thm:prop} (iii)]
The function $F$ is bounded variation, because $F$ is strictly increasing by Theorem~\ref{thm:prop} (i).
Hence, $F$ is differentiable almost everywhere on $[0,1]$ (e.g., \cite[Theorem~6.3.3]{cohn2013}).

Suppose that $x \in [0,1]$ is a differentiable point.
For any $k > 1$ we have $(y_1, \ldots, y_{k-1})$ such that 
$\sum_{i=1}^{k-1} y_i/2^i \leq x \leq ( \sum_{i=1}^{k-1} y_i/2^i ) + 1/2^k$, and
\begin{align}
\label{eq:dift}
\frac{F \left( \left( \sum_{i=1}^{k-1} (y_i/2^i) \right) + (1/2^k) \right) - F \left( \sum_{i=1}^{k-1} (y_i/2^i) \right)}{2^{-k}} 
&= 2^k r(x)_k \alpha^k.
\end{align}
Assuming that the derivative at $x$ is not zero, 
the derivative is finite and positive because $F$ is strictly increasing. 

When $y_k=0$,
\begin{align}
\label{eq:dl1}
\frac{2^{k+1} r(x)_{k+1} \alpha^{k+1}}{2^k r(x)_k \alpha^k} 
= \frac{2^{k+1} r(x)_k \alpha^{k+1}}{2^k r(x)_k \alpha^k}
= 2 \alpha.
\end{align}
When $y_{k-1}=0$ and $y_k=1$,
\begin{align}
\label{eq:dl2}
\frac{2^{k+1} r(x)_{k+1} \alpha^{k+1}}{2^k r(x)_k \alpha^k} 
= \frac{2^{k+1} ( 3 r(x)_k ) \alpha^{k+1}}{2^k r(x)_k \alpha^k} 
= 6 \alpha.
\end{align}
When $y_{k-1}=1$, $y_k=1$, and $l \geq 2$, 
where $l$ is the length of continuous $1$s including $y_k$,
\begin{align}
\frac{2^{k+1} r(x)_{k+1} \alpha^{k+1}}{2^k r(x)_k \alpha^k} 
&= \frac{2 \alpha ( a M_1^l u_0)}{a M_1^{l-1} u_0} 
= \frac{2 \alpha ( 2^{l+2}+(-1)^{l+1})}{2^{l+1}+(-1)^l} =: D_l.
\end{align}
By a simple calculation we have $\min_{l \geq 2} D_l = D_2 = 10 \alpha /3$, $\max_{l \geq 2} D_l = D_3 = 22 \alpha /5$.
Hence,
\begin{align}
\label{eq:dl3}
\frac{10 \alpha}{3} \leq \frac{2^{k+1} r(x)_{k+1} \alpha^{k+1}}{2^k r(x)_k \alpha^k} \leq \frac{22 \alpha}{5}.
\end{align}

On the other hand, because $F$ is differentiable at $x$ by Equation~(\ref{eq:dift}),
\begin{align}
& \lim_{k \to \infty} \left( ( 2^{k+1} r(x)_{k+1} \alpha^{k+1}) - (2^k r(x)_k \alpha^k) \right)
= \lim_{k \to \infty} (K_{y_{k-1}, y_k} \alpha - 1) (2^k r(x)_k \alpha^k) = 0,
\end{align}
where $K_{y_{k-1}, y_k}$ is $2$, $6$, or $10/3 \leq K_{y_{k-1}, y_k} \leq 22/5$ 
by Equations~(\ref{eq:dl1}), (\ref{eq:dl2}), and (\ref{eq:dl3}).
For any $k$ $K_{y_{k-1}, y_k} \alpha - 1 \neq 0$, and 
$\lim_{k \to \infty} (2^k r(x)_k \alpha^k)$ should be zero.

It contradicts the assumption.
We conclude that the derivative at $x$ is zero when $F$ is differentiable at $x$.
\end{proof}

\begin{proof}[Proof of Theorem~\ref{thm:prop} (iv)]
By properties of $F$ in Theorem~\ref{thm:prop} (i), (ii), and (iii)
it means that the function $F$ is a singular function.
\end{proof}

Next, we provide non-differentiable points for $F$.
If $x \in (0,1)$ is a dyadic rational $m/2^i$ 
for some $m, i \in {\mathbb Z}_{>0}$, 
$x$ is represented by a finite binary fraction. 

\begin{prop}
\label{prop:nondiff}
If $x \in (0,1)$ is represented by a finite binary fraction 
$( \sum_{i=1}^{k-1} (x_i/2^i) ) + 1/2^k$ for some $k \in {\mathbb Z}_{> 0}$, 
then $F$ is not differentiable at $x$.
\end{prop}

\begin{proof}
Let $y_m = \sum_{i=1}^{k-1} (x_i/2^i) + \sum_{i=k+1}^m (1/2^i)$ for $m > k$ and $x_i \in \{0,1\}$.
\begin{align}
\frac{F(x) - F(y_m)}{x - y_m} 
&= \frac{\sum_{i=m+1}^{\infty} r(x)_i \alpha^i}{\sum_{i=m+1}^{\infty} \frac{1}{2^i}}\\
&= 2^m r(x)_k \alpha^k \sum_{i=m-k+1}^{\infty} b_i \alpha^i\\
&= \frac{r(x)_k \alpha^k}{3} \left( \frac{2 (4 \alpha)^m}{(2 \alpha)^{k-1} (1-2 \alpha)} + \frac{(-2 \alpha)^m}{(- \alpha)^{k-1}(1+\alpha)} \right)\\
& \to + \infty \quad (m \to \infty).
\end{align}

Let $z_m = ( \sum_{i=1}^{k-1} (x_i/2^i) ) + 1/2^k 
+ ( \sum_{i=m}^{\infty} (1/2^i) )$ for $m \geq k+2$.
On the other hand,
\begin{align}
\frac{F(z_m) - F(x)}{z_m - x} &= 2^{m-1} r(z_m)_{m-1} \alpha^{m-1}\\
&= r(x)_{k+1} (2 \alpha)^{m-1} \to 0 \quad (m \to \infty).
\end{align}
Hence, $F$ is not differentiable at $x \in (0,1)$.
\end{proof}

\subsection{Functional equations for $F$}
\label{subsec:sfunc_sf}

Lastly, we give functional equations that the singular function $F$ satisfies.
Because of the self-similarity of the limit set of Rule $150$, 
the function $F$ is self-affine satisfying $F(x) = \alpha F(2x)$ for $0 \leq x \leq 1/2$, and 
$F(x) = F(x/2)/\alpha$ for $0 \leq x \leq 1$.
Including this equation, we obtain the following result.

\begin{thm}
\label{thm:sfunc}
\begin{itemize}
\item[(i)] The singular function $F$ satisfies functional equations
\begin{subnumcases}
{F(x) =}
\alpha F(2x) & \mbox{if $\displaystyle{0 \leq x < \frac{1}{2}}$}, \vspace{1mm} \label{eq:sub1} \\
3 F \left(\frac{2x-1}{2} \right) + \alpha & \mbox{if $\displaystyle{\frac{1}{2} \leq x < \frac{3}{4}}$}, \label{eq:sub2}\\
F \left(\frac{2x-1}{2} \right) + 2 F \left(\frac{4x-3}{4} \right) + \alpha + 2 \alpha^2 & \mbox{if $\displaystyle{\frac{3}{4} \leq x \leq 1}$}. \label{eq:sub3} 
\end{subnumcases}
\item[(ii)] The function $F$ is the unique continuous function on $[0,1]$ that satisfies the upper functional equations, (\ref{eq:sub1}), (\ref{eq:sub2}), and  (\ref{eq:sub3}).
\end{itemize}
\end{thm}

\begin{proof}[Proof of Theorem~\ref{thm:sfunc} (i)]
Let $x = \sum_{i=1}^{\infty} (x_i/2^i)$.
If $0 \leq x < 1/2$, then $x_1=0$, and we have
$2x = \sum_{i=1}^{\infty} (x_{i+1}/2^i)$.
\begin{align}
\alpha F(2x) &= \alpha \left( \sum_{i=1}^{\infty} x_{i+1} r(2x)_i \alpha^i \right)\\
&= \sum_{i=1}^{\infty} x_{i+1} (a M_{x_i} M_{x_{i-1}} \cdots M_{x_2} u_0) \alpha^{i+1}\\
&= \sum_{i=2}^{\infty} x_i (a M_{x_{i-1}} M_{x_{i-2}} \cdots M_{x_2} u_0) \alpha^i = F(x).
\end{align}

If $1/2 \leq x < 3/4$, then $x_1=1$ and $x_2=0$.
Since $(2x-1)/2 = \sum_{i=3}^{\infty} (x_i/2^i)$,  
we have $F((2x-1)/2) = \sum_{i=3}^{\infty} x_i r((2x-1)/2)_i \alpha^i$.
Thus, 
\begin{align}
F(x) - F \left(\frac{2x-1}{2} \right) 
&= \left( \alpha + \sum_{i=3}^{\infty} x_i r(x)_i \alpha^i \right) - \sum_{i=3}^{\infty} x_i r \left(\frac{2x-1}{2} \right)_i \alpha^i \\
&= \alpha + \sum_{i=3}^{\infty} x_i \left(  a M_{x_{i-1}} \cdots M_{x_2} ( M_{x_1} - M_0) u_0 \right) \alpha^i\\
&=  \alpha + 2 \sum_{i=3}^{\infty} x_i \left(  a M_{x_{i-1}} \cdots M_{x_2} u_0 \right) \alpha^i \label{eq:x20}\\
&=  \alpha + 2 F \left(\frac{2x-1}{2} \right).
\end{align}
Equation~(\ref{eq:x20}) follows directly from $x_2 = 0$.
Hence, we have $F(x) = 3 F((2x-1)/2) + \alpha$.

If $3/4 \leq x \leq 1$, then $x_1=1$ and $x_2=1$.
Since $(2x-1)/2 = \sum_{i=2}^{\infty} (x_i/2^i)$,  
we have $F((2x-1)/2) = \sum_{i=2}^{\infty} x_i r((2x-1)/2)_i \alpha^i$, 
and since $(4x-3)/4 = \sum_{i=3}^{\infty} (x_i/2^i)$,  
we have $F((4x-3)/4) = \sum_{i=3}^{\infty} x_i r((4x-3)/4)_i \alpha^i$.
\begin{align}
F(x) - \alpha - 2 \alpha^2 
&= \left( \alpha + 3 \alpha^2 + \sum_{i=3}^{\infty} x_i r(x)_i \alpha^i \right) -\alpha - 2 \alpha^2 \\
&= \alpha^2 + \sum_{i=3}^{\infty} x_i (a M_{x_{i-1}} \cdots M_{x_3} M_{x_2} M_{x_1} u_0) \alpha^i\\
&= \alpha^2 + \sum_{i=3}^{\infty} x_i (a M_{x_{i-1}} \cdots M_{x_3} (M_1 + 2 I_2) u_0) \alpha^i\\
&= \sum_{i=2}^{\infty} x_i (a M_{x_{i-1}} \cdots M_{x_2} u_0) \alpha^i + 2 \sum_{i=3}^{\infty} x_i (a M_{x_{i-1}} \cdots M_{x_3} u_0) \alpha^i\\
&= F \left(\frac{2x-1}{2} \right) + 2 F \left(\frac{4x-3}{4} \right),
\end{align}
where $I_2 = 
\left( 
\begin{array}{cc}
1 & 0 \\
0 & 1 
\end{array}
\right)$.
\end{proof}

\begin{proof}[Proof of Theorem~\ref{thm:sfunc} (ii)]
The uniqueness is obtained by showing that the functional equations determine the value 
for each dyadic rational on $[0,1]$.

We first obtain $F(0)=0$ by Equation~(\ref{eq:sub1}), and $F(1/2) = \alpha$ by Equation~(\ref{eq:sub2}).
Thus $F(1/2^i) = \alpha^i$ for $i \in {\mathbb Z}_{>0}$ by Equation~(\ref{eq:sub1}), 
and $F(1)=F(1/2^0)=1$ by Equation~(\ref{eq:sub3}).
Calculating $F(3/4)$, we obtain $F(3/2^i)$ for $i \geq 2$, and calculating $F(5/8)$ and $F(7/8)$, we obtain $F(5/2^i)$ and $F(7/2^i)$ for $i \geq 3$.
Accordingly, iterating the same procedure for $i \geq 4$, we determine 
$F(m/2^i)$ for each dyadic rational point $m/2^i$ on $[0,1]$.

Since the dyadic rationals on $[0,1]$ are dense, there is a unique continuous function.
\end{proof}

\section{Conclusions and future works}
\label{sec:cr}

In this paper, a function on the unit interval has given by the dynamic of the one-dimensional elementary cellular automaton Rule $150$.
Since the limit set of Rule $150$ holds a self-similarity, the resulting function is self-affine.
We have shown that the function is strictly increasing, uniformly continuous, and differentiable almost everywhere, that means it is a singular function.
We also have given the functional equations that the singular function satisfies, and we have proven that the function is the only solution of the functional ones.

First future work is to show at which points of $[0,1]$ the singular function $F$ is differentiable.
For some other singular functions their differentiabilities at all points have been revealed.
About the Cantor function, for points outside of the Cantor set the derivative is zero, and for points in the Cantor set except $\{0, 1\}$ the derivative is infinity.
In the case of Salem's function, the differentiability is determined by the parameter of the function and the density of the digit $0$ or $1$ in the binary expansion of $x \in [0,1]$ \cite{kawamura2011}.
By Proposition~\ref{prop:nondiff} we already have shown that at finite binary fraction points in $[0, 1]$ the function $F$ is not differentiable, 
however, we did not mention the differentiability at the other points. 

Second future work is about the relationship between the other cellular automata and singular functions.
The authors already obtained the relationships among Rule $90$, a two-dimensional automaton, and Salem's function \cite{kawanami2014, kawa2014a, kawanami2017, kawanami2019}.
In addition, in this paper we have studied it between Rule $150$ and the new singular function.
Since dynamics of cellular automata often hold self-similarity,
it is expected that there exist some relationship with a function satisfying self-similarity. 
We are going to search for and study the other automata and functions.

\section*{Acknowledgment}
This work is partly supported by the Grant in Aid for Scientific Research 18K13457.\\


\begin{thebibliography}{10}

\bibitem{weier1872}
K. Weierstrass.
\newblock {\"U}ber continuirliche functionen eines reellen arguments, die
  f{\"u}r keinen werth des letzeren einen bestimmten differentialquotienten
  besitzen.
\newblock {\em Mathematische Werke}, 2:71--74, 1872, \url{https://doi.org/10.1007/978-3-322-91273-2_5}.
\newblock (English translation in \cite{weier1993}, pages 3--9).

\bibitem{takagi1903}
T. Takagi.
\newblock A simple example of the continuous function without derivative.
\newblock {\em Proceedings of the Physico-Mathematical Society of Japan},
  1:176--177, 1903, \url{https://doi.org/10.11429/subutsuhokoku1901.1.F176}.

\bibitem{hatayama1984}
M. Hata and M. Yamaguti.
\newblock The {T}akagi function and its generalization.
\newblock {\em Japan journal of applied mathematics}, 1(1):183--199, 1984, \url{https://doi.org/10.1007/BF03167867}.

\bibitem{okamoto2005}
H. Okamoto.
\newblock A remark on continuous, nowhere differentiable functions.
\newblock {\em Proceedings of the Japan Academy, Series A, Mathematical
  Sciences}, 81(3):47--50, 03 2005, \url{https://doi.org/10.3792/pjaa.81.47}.

\bibitem{okamoto2007}
H. Okamoto and M. Wunsch.
\newblock A geometric construction of continuous, strictly increasing singular
  functions.
\newblock {\em Proceedings of the Japan Academy, Series A, Mathematical
  Sciences}, 83(7):114--118, 07 2007, \url{https://doi.org/10.3792/pjaa.83.114}.

\bibitem{kobayashi2009}
K. Kobayashi.
\newblock On the critical case of okamoto's continuous non-differentiable
  functions.
\newblock {\em Proceedings of the Japan Academy, Series A, Mathematical
  Sciences}, 85(8):101--104, 10 2009, \url{https://doi.org/10.3792/pjaa.85.101}.

\bibitem{cantor1884}
G. Cantor.
\newblock De la puissance des ensembles parfaits de points: Extrait d'une
  lettre adress{\'e}e {\`a} l'{\'e}diteur [the power of perfect sets of points:
  Extract from a letter addressed to the editor].
\newblock {\em Acta Mathematica}, 4:381--392, 1884, \url{https://doi.org/10.1007/BF02418423}.

\bibitem{salem1943}
R. Salem.
\newblock On some singular monotonic functions which are strictly increasing.
\newblock {\em Transactions of the American Mathematical Society}, 53:427--439,
  1943, \url{https://doi.org/10.2307/1990210}.

\bibitem{derham1957}
G. de~Rham.
\newblock Sur quelques courbes definies par des equations fonctionnelles.
\newblock {\em Rendiconti del Seminario Matematico Universit{\`a} e Politecnico
  di Torino}, 16:101--113, 1957.

\bibitem{ulam1934}
Z.~A. Lomnicki and S. Ulam.
\newblock Sur la th{\'e}orie de la mesure dans les espaces combinatoires et son
  application au calcul des probabilit{\'e}s i. variables ind{\'e}pendantes.
\newblock {\em Fundamenta Mathematicae}, 23:237--278, 1934, \url{https://doi.org/10.4064/fm-23-1-237-278}.

\bibitem{yhk1997}
M. Yamaguti, M. Hata, and J. Kigami.
\newblock {\em Mathematics of fractals, Translations of Mathematical
  Monographs}.
\newblock American Mathematical Society, 1997.
\newblock (translated by K. Hudson).

\bibitem{kawanami2014}
A. Kawaharada and T. Namiki.
\newblock Cumulative distribution of rule 90 and lebesgue's singular function.
\newblock {\em Proceedings of AUTOMATA 2014}, pages 165--169, 2014.

\bibitem{kawa2014a}
A. Kawaharada.
\newblock Fractal patterns created by {U}lam's cellular automaton.
\newblock {\em Proceedings of International Workshop on Applications and
  Fundamentals of Cellular Automata 2014}, pages 484--486, 2014, \url{https://doi.org/10.1109/CANDAR.2014.51}.

\bibitem{kawanami2017}
A. Kawaharada and T. Namiki.
\newblock Fractal structure of a class of two-dimensional two-state cellular
  automata.
\newblock {\em Proceedings of International Workshop on Applications and
  Fundamentals of Cellular Automata 2017}, pages 205--208, 2017, \url{https://doi.org/10.1109/CANDAR.2017.89}.

\bibitem{kawanami2019}
A. Kawaharada and T. Namiki.
\newblock Relation between spatio-temporal patterns generated by
  two-dimensional cellular automata and a singular function.
\newblock {\em International Journal of Networking and Computing},
  9(2):354--369, 2019, \url{https://doi.org/10.15803/ijnc.9.2_354}.

\bibitem{kawanami2020}
A. Kawaharada and T. Namiki.
\newblock Number of nonzero states in prefractal sets generated by cellular
  automata.
\newblock {\em Journal of Mathematical Physics}, pages 1--11, in press.

\bibitem{takahashi1992}
S. Takahashi.
\newblock Self-similarity of linear cellular automata.
\newblock {\em Journal of Computer and System Sciences}, 44:114--140, 1992, \url{https://doi.org/10.1016/0022-0000(92)90007-6}.

\bibitem{claussen2008}
J.~Christian Claussen.
\newblock Time evolution of the rule 150 cellular automaton activity from a
  fibonacci iteration.
\newblock {\em Journal of Mathematical Physics}, 49(062701):1--12, 2008, \url{https://doi.org/10.1063/1.2939398}.

\bibitem{cohn2013}
D.~L. Cohn.
\newblock {\em Measure Theory}.
\newblock Birkh{\"a}user Basel, second edition edition, 2013.

\bibitem{kawamura2011}
K. Kawamura.
\newblock On the set of points where {L}ebesgue's singular function has the
  derivative zero.
\newblock {\em Proceedings of the Japan Academy, Series A, Mathematical
  Sciences}, 87(9):162--166, 2011, 
  \url{https://doi.org/10.3792/pjaa.87.162}.

\bibitem{weier1993}
G.~A Edgar, editor.
\newblock {\em Classics on Fractals, Studies in Nonlinearity}.
\newblock Addison-Wesley Publishing Company, 1993.

\end{thebibliography}

\end{document}